\documentclass[11pt]{article}

\usepackage{epsfig}
\usepackage{graphicx}
\usepackage{amsbsy}
\usepackage{amsmath}
\usepackage{amsfonts}
\usepackage{amssymb}
\usepackage{textcomp}
\usepackage{hyperref}
\usepackage{aliascnt}

\newcommand{\mcm}[3]{\newcommand{#1}[#2]{{\ensuremath{#3}}}} 

\mcm{\tuple}{1}{\langle #1 \rangle}
\mcm{\name}{1}{\ulcorner #1 \urcorner}
\mcm{\Nbb}{0}{\mathbb{N}}
\mcm{\Zbb}{0}{\mathbb{Z}}
\mcm{\Rbb}{0}{\mathbb{R}}
\mcm{\Cbb}{0}{\mathbb{C}}
\mcm{\Qbb}{0}{\mathbb{Q}}
\mcm{\Acal}{0}{\cal A}
\mcm{\Bcal}{0}{\cal B}
\mcm{\Ccal}{0}{\cal C}
\mcm{\Dcal}{0}{\cal D}
\mcm{\Ecal}{0}{\cal E}
\mcm{\Fcal}{0}{\cal F}
\mcm{\Gcal}{0}{\cal G}
\mcm{\Hcal}{0}{\cal H}
\mcm{\Ical}{0}{\cal I}
\mcm{\Jcal}{0}{\cal J}
\mcm{\Kcal}{0}{\cal K}
\mcm{\Lcal}{0}{\cal L}
\mcm{\Mcal}{0}{\cal M}
\mcm{\Ncal}{0}{\cal N}
\mcm{\Ocal}{0}{{\cal O}}
\mcm{\Pcal}{0}{{\cal P}}
\mcm{\Qcal}{0}{{\cal Q}}
\mcm{\Rcal}{0}{{\cal R}}
\mcm{\Scal}{0}{{\cal S}}
\mcm{\Tcal}{0}{{\cal T}}
\mcm{\Ucal}{0}{{\cal U}}
\mcm{\Vcal}{0}{{\cal V}}
\mcm{\Wcal}{0}{{\cal W}}
\mcm{\Xcal}{0}{{\cal X}}
\mcm{\Ycal}{0}{{\cal Y}}
\mcm{\Zcal}{0}{{\cal Z}}
\mcm{\Mfrak}{0}{\mathfrak M}

\mcm{\restric}{0}{\upharpoonright}
\mcm{\upset}{0}{\uparrow}
\mcm{\onto}{0}{\twoheadrightarrow}
\mcm{\smallNbb}{0}{{\small \mathbb{N}}}
\DeclareMathOperator{\preop}{op}
\mcm{\op}{0}{^{\preop}}

{\begin{array}{c}
\setlength{\unitlength}{1em}}%
{\end{array}}

\usepackage{amsthm}

\newcommand{\theoremize}[2]{\newaliascnt{#1}{thm} \newtheorem{#1}[#1]{#2} \aliascntresetthe{#1}}

\theoremstyle{plain}

\theoremize{lem}{Lemma}
\theoremize{skolem}{Skolem}
\theoremize{fact}{Fact}
\theoremize{sublem}{Sublemma}
\theoremize{claim}{Claim}
\theoremize{obs}{Observation}
\theoremize{prop}{Proposition}
\theoremize{cor}{Corollary}
\theoremize{que}{Question}
\theoremize{oque}{Open Question}
\theoremize{con}{Conjecture}

\theoremstyle{definition}
\theoremize{dfn}{Definition}
\theoremize{rem}{Remark}
\theoremize{eg}{Example}
\theoremize{exercise}{Exercise}
\theoremstyle{plain}

\usepackage{verbatim}
\usepackage{enumerate}
\usepackage[all]{xy}

\usepackage{subfig}

\title{On Andreae's Ubiquity Conjecture}

\author{Johannes Carmesin\footnote{Funded by EPSRC, grant number EP/T016221/1}
\medskip
\\
  {University of Birmingham}
}
\newcommand{\sm}{\setminus}

\mcm{\Fbb}{0}{\mathbb{F}}

\begin{document}

\maketitle

\abstract{A graph $H$ is \emph{ubiquitous} if for every graph $G$ that for every
natural number $n$ contains $n$ vertex-disjoint $H$-minors contains
infinitely many vertex-disjoint $H$-minors.
Andreae conjectured that every locally finite graph is ubiquitous. We
give a disconnected counterexample to this conjecture.
It remains open
whether every \emph{connected} locally finite graph is ubiquitous.}

\vspace{.3cm}

\noindent {\bf Introduction.}
Thomas conjectured that the class of countable graphs is well-quasi
ordered by the minor relation \cite{T:wqo_planar}. Related to this, in 2002 Andreae conjectured that every locally finite\footnote{Andreae constructed an uncountable graph that is not ubiquitous.}
graph is ubiquitous; here a graph $H$
is  \emph{ubiquitous}  (with respect to the minor relation) if for every graph $G$ that for every natural number $n$ contains $n$ vertex-disjoint $H$-minors contains
infinitely many vertex-disjoint $H$-minors \cite{andreae02}.
Finite graphs are clearly ubiquitous, and by Halin's theorem \cite{halin65}
the ray is ubiquitous. Andreae proved quite a few cases of his
conjecture, particularly that connected graphs of finite tree-width such that all their maximal 2-connected subgraphs are finite are ubiquitous \cite{andreae13}.
Recently Bowler, Elbracht, Erde, Gollin, Heuer, Pitz and Teegen put forward a series of papers \cite{{ubi1},{ubi2},{ubi3}}
in which they prove that a large class of locally finite graphs are
ubiquitous, including the 2-dimensional grid!
Here we note that the condition of \lq connectedness\rq\ needs to be
added to Andreae's conjecture. In other words, we construct a
disconnected locally finite graph that is not ubiquitous.

\begin{figure}[ht]
\centering
\includegraphics[height=8\baselineskip]{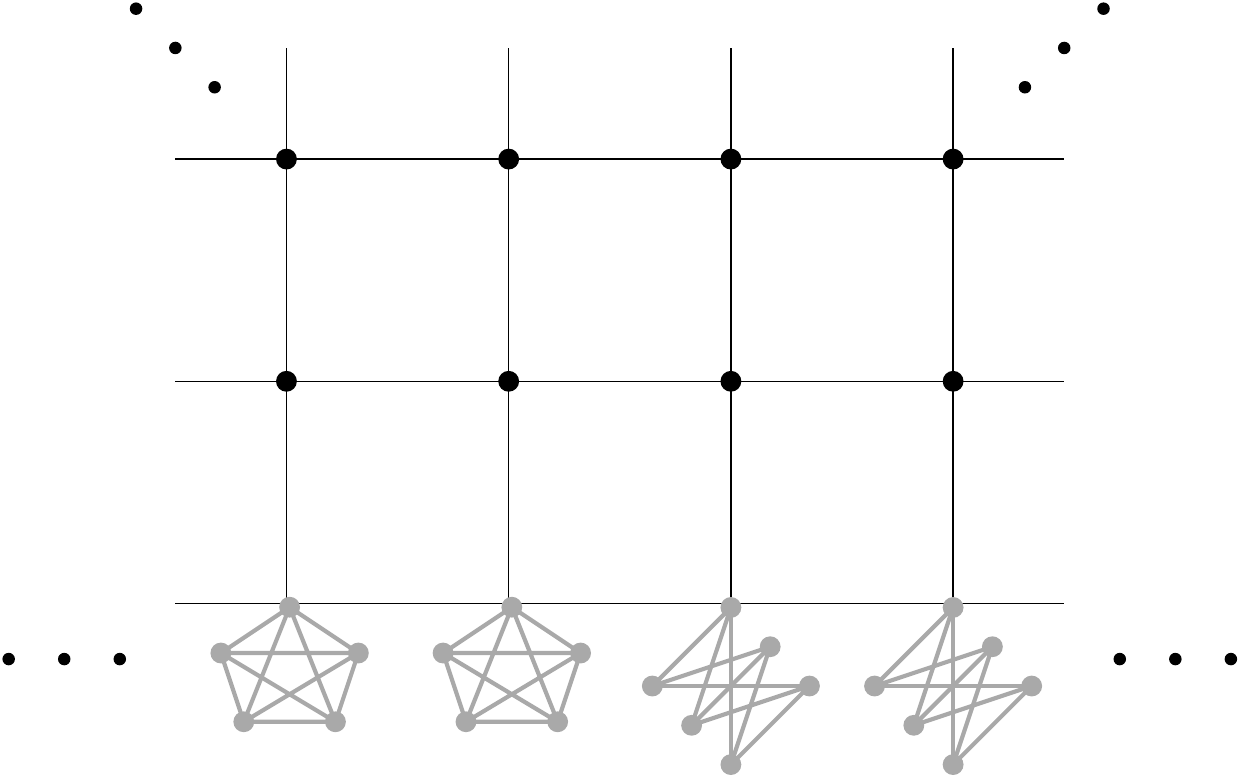}
\caption{The graph $G$ indicated in this figure is obtained from the
half-grid $\Zbb\times \Nbb$ by attaching
a $K_5$ at all vertices $(a,0)$ with $a<0$ and a $K_{3,3}$ at all vertices $(b,0)$ with $b\geq 0$.}
\label{fig:G}
\end{figure}

\noindent{\bf The construction.}
The graph $G$ is indicated in \autoref{fig:G} and defined in its caption.
Let $I$ be the 1-sum\footnote{The \emph{1-sum} of $K_5$ and $K_{3,3}$ is obtained from their disjoint by identifying two vertices, one from each.} of $K_5$ and $K_{3,3}$.
Let $H$ be the disjoint union of $I$ with a ray.

\theoremstyle{plain}
\newtheorem{slemma}{Lemma}

\begin{slemma}\label{nice}
In the graph $G$, each $I$-minor is composed of a $K_5$-subgraph attached at a vertex $(a,0)$, a path in the $\Zbb\times \Nbb$ grid from  $(a,0)$ to a vertex $(b,0)$ with $b\geq 0$ and a $K_{3,3}$ attached at $(b,0)$.
\end{slemma}

\begin{proof}
The family of maximally 2-connected subgraphs of $G$ consists of the $\Zbb\times \Nbb$ grid together with the attached $K_5$-subgraphs and the attached $K_{3,3}$-subgraphs. Thus by Kuratowski's theorem, the attached $K_5$-subgraphs are the only $K_5$-minors of $G$. Hence every $I$-minor consists of an attach $K_5$-subgraph, an attached $K_{3,3}$-subgraph and a path between them.
\end{proof}

\theoremstyle{plain}
\newtheorem{propo}{Proposition}
\addtocounter{propo}{1}

\begin{propo}
For every natural number $n$, the graph $G$ has $n$ vertex-disjoint
$H$-minors, yet it does not have infinitely many vertex-disjoint
$H$-minors.
\end{propo}

\begin{proof}
The graph $G$ clearly has infinitely many vertex-disjoint
$I$-minors. Deleting finitely many of these finite minors, leaves one component in
which the thick end of $G$ lives. This component contains infinitely
many vertex-disjoint rays by Halin's theorem. Thus for every natural
number $n$, the graph $G$ has $n$ vertex-disjoint $H$-minors.

Now let $\Ccal$ be a family of infinitely many vertex-disjoint
$I$-minors of $G$. By Lemma \ref{nice}, each of them contains a path $P_b$ from a vertex $(a,0)$ with $a<0$ to a vertex $(b,0)$ with $b\geq 0$.
So these paths form an infinite family $(P_b|b\in B)$ of vertex-disjoint paths in $G$.
As $G$ is 1-ended, the graph $G\sm P_b$ has exactly one infinite component $C_b$; this component contains all paths $P_c$ with $c\geq b$.
Thus the family $(C_b|b\in B)$ is decreasing and it is strictly decreasing as the boundaries are the paths $P_b$, which are vertex-disjoint.
Hence the intersection $\bigcap_{b\in B} C_b$ is empty.
As every ray of $G\sm \bigcup_{b\in B} P_b$ must be contained in each component $C_b$, the graph $G\sm \bigcup_{b\in B} P_b$ is rayless.
Hence $G$
does not have infinitely many vertex-disjoint $H$-minors.
\end{proof}

\section{Acknowledgement}

I am grateful to Alex Scott, who found an error in an earlier version of this manuscript.

\bibliographystyle{plain}
\bibliography{literatur}

\begin{thebibliography}{1}

\bibitem{andreae02}
Thomas Andreae.
\newblock On disjoint configurations in infinite graphs.
\newblock {\em Journal of Graph Theory}, 39(4):222--229, 2002.

\bibitem{andreae13}
Thomas Andreae.
\newblock Classes of locally finite ubiquitous graphs.
\newblock {\em Journal of Combinatorial Theory, Series B}, 103(2):274 -- 290,
  2013.

\bibitem{ubi3}
Nathan Bowler, Christian Elbracht, Joshua Erde, J.~Pascal Gollin, Karl Heuer,
  Max Pitz, and Maximilian Teegen.
\newblock Ubiquity in graphs {I}: Topological ubiquity of trees.
\newblock {A}vailable at ArXiv.org.

\bibitem{ubi2}
Nathan Bowler, Christian Elbracht, Joshua Erde, J.~Pascal Gollin, Karl Heuer,
  Max Pitz, and Maximilian Teegen.
\newblock Ubiquity in graphs {II}: Ubiquity of graphs with nowhere-linear end
  structure.
\newblock {A}vailable at ArXiv.org.

\bibitem{ubi1}
Nathan Bowler, Christian Elbracht, Joshua Erde, J.~Pascal Gollin, Karl Heuer,
  Max Pitz, and Maximilian Teegen.
\newblock Ubiquity in graphs {III}: Ubiquity of locally finite graphs with
  extensive tree-decompositions.
\newblock {A}vailable at ArXiv.org.

\bibitem{halin65}
R.~Halin.
\newblock {\"U}ber die {M}aximalzahl fremder unendlicher {W}ege in {G}raphen.
\newblock {\em Math.\ Nachr.}, 30:63--85, 1965.

\bibitem{T:wqo_planar}
Robin Thomas.
\newblock Well-quasi-ordering infinite graphs with forbidden finite planar
  minor.
\newblock {\em Trans. Amer. Math. Soc.}, 312:279--313, 1989, no. 1.

\end{thebibliography}

\end{document}